\documentclass{amsproc}
\usepackage{amssymb, amsfonts, amsbsy, latexsym, epsfig, color}

\newtheorem{Thm}{Theorem}[section]

\newtheorem{proposition}[Thm]{Proposition}
\newtheorem{definition}[Thm]{Definition}
\newtheorem{remark}[Thm]{Remark}

\newtheorem{example}[Thm]{Example}

\newtheorem{theorem}[Thm]{Theorem}

\newtheorem{problem}[Thm]{Problem}

\newcommand{\norm}[1]{\|#1\|}

\DeclareMathOperator{\lspan}{span}

\numberwithin{equation}{section}

\newcommand{\abs}[1]{\lvert#1\rvert}

\begin{document}

\title{Phase retrieval and norm retrieval}

\author[Bahmanpour]{Saeid Bahmanpour}
\address{Department of Mathematics, University
of Missouri, Columbia, MO 65211-4100}
\email{sbgxf@math.missouri.edu}

\author[Cahill]{Jameson Cahill}
\address{Mathematics Department, Duke University, Box 90320,
Durham, NC 27708-0320}
\email{jameson.cahill@gmail.com}

\author[Casazza]{Peter G. Casazza}
\address{Department of Mathematics, University
of Missouri, Columbia, MO 65211-4100}
\email{casazzap@missouri.edu}

\author[Jasper]{John Jasper}
\address{Department of Mathematics, University
of Oregon, Eugene, OR 97403}
\email{jjasper30@gmail.com}

\author[Woodland]{Lindsey M. Woodland}
\address{Department of Mathematics, University
of Missouri, Columbia, MO 65211-4100}
\email{lmwvh4@mail.missouri.edu}

\thanks{Bahmanpour, Cahill, Casazza, Jasper and Woodland were supported by
 NSF 1307685; and  NSF ATD 1042701; AFOSR DGE51: FA9550-11-1-0245}

\subjclass{Primary 32C15 }
\date{today}

\keywords{Phase retrieval, norm retrieval, frames}

\begin{abstract}
Phase retrieval has become a very active area of research.
We will classify when phase retrieval by Parseval
frames passes to the Naimark complement and when phase
retrieval by projections passes to the orthogonal 
complements.  We introduce a new concept we
call norm retrieval and show that this is what is necessary
for passing phase retrieval to complements.  This leads to
a detailed study of norm retrieval and its relationship
to phase retrieval.  One fundamental result:  a frame
$\{\varphi_i\}_{i=1}^M$ yields phase retrieval if and only if
$\{T\varphi_i\}_{i=1}^M$ yields norm retrieval for every
invertible operator $T$. 
\end{abstract}

\maketitle


\section{Introduction}


Phase retrieval is the problem of recovering a signal from the absolute value of linear measurement coefficients called intensity measurements. Since its mathematical introduction in \cite{BCE}, phase retrieval has become a very active area of research. Often times in engineering problems, during processing the phase of a signal is lost and hence a method for recovering this signal with this lack of information is necessary. In particular, this problem occurs in speech recognition \cite{BR,RJ,RBSC}, X-ray crystallography and electron microscopy \cite{BM,F78,F82} and a number of other areas. In some scenarios, such as crystal twinning \cite{D}, the signal is projected onto higher dimensional subspaces and one has to recover this signal from the norms of these projections. This is called {\it phase retrieval by projections} and a deep study of this area was made in \cite{CCPW} (see also \cite{BE}). In \cite{CCPW}, they analyze and compare phase retrieval by vectors to phase retrieval by projections. Although they solve numerous problems in this area, there are still many open fundamental questions about phase retrieval in both cases.

In \cite{CCPW} it was shown that if a family of projections $\{P_i\}_{i=1}^M$ yields phase retrieval, it need not occur that $\{(I-P_i)\}_{i=1}^M$ yields phase retrieval. In the present paper we will classify when this occurs and solve the related problem: If a Parseval frame yields phase retrieval, when does its Naimark complement yield phase retrieval? The fundamental notion which connects phase retrieval to complements is something we call {\it norm retrieval}. After showing that norm retrieval is central to these questions, we make a detailed study of norm retrieval and its relationship to phase retrieval.  In particular, we show that a collection of vectors $\{\varphi_i\}_{i=1}^M$ yields phase retrieval if and only if  $\{T\varphi_i\}_{i=1}^M$ yields norm retrieval for every invertible operator $T$. Another fundamental idea which connects these results is the question of when the identity operator is in the span of $\{P_i\}_{i=1}^M$. Often times when a collection of vectors or projections yield phase retrieval then the identity is in their span. Moreover, in \cite{CCPW} they show that having the identity in the span of a family of projections doing phase retrieval, will generally yield that the orthogonal complements do phase retrieval. We analyze this question regarding the identity operator and its association to norm retrieval and thus phase retrieval.  We will also give a number of examples throughout showing that these results are best possible.  For an up to date review of phase retrieval by vectors and projections please see \cite{CW}.


\section{Preliminaries}

In this section we introduce some necessary definitions and basic theorems from finite frame theory.  Throughout, let $\mathcal{H}_N$ denote an $N$-dimensional Hilbert space.

\begin{definition}
A family of vectors $\{\varphi_i\}_{i=1}^M$ in $\mathcal{H}_N$ is a \textbf{frame} if there are constants $0 < A \leq B < \infty$ so that for all $x \in \mathcal{H}_N$,
\[
A \| x \|^2 \leq \sum_{i=1}^M |\langle x,\varphi_i\rangle|^2 \leq B \|x\|^2,
\]
where $A$ and $B$ are the \textbf{lower and upper frame bounds}, respectively.  If we can choose $A=B$ then the frame is called \textbf{tight}, and if $A=B=1$ it is called a \textbf{Parseval frame}.
\end{definition}
\noindent Note that in the finite dimensional setting, a frame is simply a spanning set of vectors in the Hilbert space.

If $\{\varphi_i\}_{i=1}^M$ is a frame for $\mathcal{H}_N$, then the \textbf{analysis operator} of the frame is the operator $T : \mathcal{H}_N \to \ell_2(M)$ given by
\[
T(x) = \{ \langle x,\varphi_i\rangle \}_{i=1}^M
\]
and the \textbf{synthesis operator} is the adjoint operator, $T^*$, which satisfies
\[
T^{*}\left( \{a_i\}_{i=1}^M \right) = \sum_{i=1}^M a_i \varphi_i.
\]
The \textbf{frame operator} is the positive, self-adjoint, invertible operator $S := T^* T$ on $\mathcal{H}_N$ and satisfies:
\[
S(x) = T^*T(x) = \sum_{i=1}^M \langle x,\varphi_i\rangle \varphi_i.
\]
Moreover, $\{\varphi_i\}_{i=1}^M$ is a frame if there are constants $0<A\leq B<\infty$ such that its frame operator $S$ satisfies $AI \leq S \leq BI$ where $I$ is the identity on $\mathcal{H}_N$.

In particular, the frame operator of a Parseval frame is the identity operator. This fact makes Parseval frames very helpful in applications because they possess the property of perfect reconstruction. That is, if $\{\varphi_i\}_{i=1}^M$ is a Parseval frame for $\mathcal{H}_N$, then for any $x\in \mathcal{H}_N$ we have 
\[x=\sum_{i=1}^M\langle x,\varphi_i\rangle \varphi_i.\]

There is a direct method for constructing Parseval frames.  For $M \geq N$, given an $M \times M$ unitary matrix, if we select any $N$ rows from this matrix, then the column vectors from these rows form a Parseval frame for $\mathcal{H}_N$. Moreover, the leftover set of $M-N$ rows, also have the property that its $M$ columns form a Parseval frame for $\mathcal{H}_{M-N}$. The next theorem, known as Naimark's Theorem, says that this is the only way to obtain Parseval
frames.

\begin{theorem}[Naimark's Theorem]\cite{CK}\label{naimark}
Let $\Phi=\{\varphi_i \}_{i=1}^M$ be a frame for $\mathcal{H}_N$ with analysis operator $T$, let $\{e_i\}_{i=1}^M$ be the standard basis of $\ell_2\left(M\right)$, and let $P:\ell_2\left(M\right) \rightarrow \ell_2\left(M\right)$ be the orthogonal projection onto $\mbox{range}\left(T\right)$. Then the following conditions are equivalent:
\begin{enumerate}
\item $\{\varphi_i\}_{i=1}^M$ is a Parseval frame for $\mathcal{H}_N$.
\item For all $i=1,\dots,M$, we have $Pe_i=T\varphi_i$.
\item There exist $\psi_1,\dots, \psi_M \in \mathcal{H}_{M-N}$ such that $\{ \varphi_i \oplus \psi_i \}_{i=1}^M$ is an orthonormal basis of $\mathcal{H}_M$.
\end{enumerate}

Moreover, $\{\psi_i\}_{i=1}^M$ is a Parseval frame for $\mathcal{H}_{M-N}$. 
\end{theorem}

Explicitly, we call $\{\psi_i\}_{i=1}^M$ the {\bf Naimark complement} of $\Phi$.  If $\Phi = \{\varphi_i\}_{i=1}^M$
is a Parseval frame, then the analysis operator $T$ of the
frame is an isometry.  So we can associate 
$\varphi_i$ with $T\varphi_i = Pe_i$, and with a slight
 abuse of notation
we have:

\begin{theorem}[Naimark's Theorem]
$\Phi = \{\varphi_i\}_{i=1}^M$ is a Parseval frame for
$\mathcal{H}_N$ if and only if there is an $M$-dimensional Hilbert space $\mathcal{K}_M$ with an orthonormal
basis $\{e_i\}_{i=1}^M$ such that the orthogonal projection
$P:\mathcal{K}_M \rightarrow \mathcal{H}_N$ satisfies
$Pe_i= \varphi_i$ for all $i=1,\dots,M$.  Moreover, the Naimark complement of
$\Phi$ is $\{(I-P)e_i\}_{i=1}^M.$
\end{theorem}
 
Note that Naimark complements are only defined for Parseval frames.  Furthermore, Naimark complements  are only defined up to unitary equivalence, that is, if $\{\varphi_i\}_{i=1}^M\subseteq\mathcal{H}_N$ and $\{\psi_i\}_{i=1}^M\subseteq\mathcal{H}_{M-N}$ are Naimark complements, and $U$ and $V$ are unitary operators, then $\{U\varphi_i\}_{i=1}^M$ and $\{V\psi_i\}_{i=1}^M$ are also Naimark complements.

Given a sequence of vectors $\{\varphi_i\}_{i=1}^M$ and an orthogonal projection $P$, throughout this paper we will frequently refer to $\{P\varphi_i\}_{i=1}^M$ as possessing certain properties, such as: full spark, complement property, phase retrieval, and so on. By this we mean that it has these properties in the range of $P$.

This concludes a brief introduction to finite frames and other terms which are necessary throughout the present paper. For a more in depth review of finite frame theory, the interested reader is referred to \cite{CK}.


\section{Phase retrieval by vectors and Naimark complements}

Signal reconstruction is an important topic of research with applications to numerous fields. Recovering the phase of a signal through the use of a redundant system of vectors (i.e. a frame) or a spanning collection of subspaces is currently a well studied topic (see \cite{CW} for a survey of this 
field).

\begin{definition} 
A set of vectors $\{\varphi_i\}_{i=1}^M$ in $\mathbb{R}^N$ (or $\mathbb{C}^N$) yields \textbf{phase retrieval} if for all $x, y \in \mathbb{R}^N$ (or $\mathbb{C}^N$) satisfying  $|\langle x,\varphi_i\rangle|=|\langle y, \varphi_i\rangle |$ for all $i=1,\dots, M$, then $x=cy$ where $c = \pm 1$ in $\mathbb{R}^N$ (and for $\mathbb{C}^N$, $c \in \mathbb{T}^1$ where $\mathbb{T}^1$ is the complex unit circle). 
\end{definition}

A fundamental result from \cite{BCE} classifies
 phase retrieval in the real case by way of the {\bf complement
property}.

\begin{definition}
A frame $\{\varphi_i\}_{i=1}^M$ in $\mathcal{H}_N$ satisfies the \textbf{complement property} if for all subsets $\mathcal{I} \subset \{1,\dots, M\}$, either $\lspan\{\varphi_i\}_{i\in \mathcal{I}}=\mathcal{H}_N$ or $\lspan\{\varphi_i\}_{i\in \mathcal{I}^c}=\mathcal{H}_N$.
\end{definition}

Similar to the complement property, the notion of {\bf full spark} for a set of vectors is very
useful for guaranteeing that large subsets of the vectors
actually span the space.

\begin{definition}
Given a family of vectors $\Phi=\{\varphi_i\}_{i=1}^M$ in $\mathcal{H}_N$, the \textbf{spark} of $\Phi$ is defined as the cardinality of the smallest linearly dependent subset of $\Phi$.  When $\mbox{spark}(\Phi)=N+1$, every subset of size $N$ is linearly independent, and $\Phi$ is said to be {\bf full spark}.
\end{definition}

\begin{theorem}\label{thmbce}\cite{BCE}
A frame $\{\varphi_i\}_{i=1}^M$ in $\mathbb{R}^N$ yields
phase retrieval if and only if it has the complement
property.  In particular, a full spark frame with
$2N-1$ vectors yields phase retrieval.  Moreover,
if $\{\varphi_i\}_{i=1}^M$
yields phase retrieval in $\mathbb{R}^N$, then $M\ge 2N-1$,
and no set of $2N-2$ vectors yields phase retrieval.
\end{theorem}

In general, it is not necessary for a frame to be full spark in order to yield phase retrieval.  For example, as long as our frame
contains a full spark subset of $2N-1$ vectors, it will
do phase retrieval.  However, if the frame contains exactly 
$2N-1$ vectors, then clearly it does phase retrieval if
and only if it is full spark.

From Theorem \ref{thmbce}, it is clear that full spark and the complement property are important properties for a frame when it comes to classification results regarding phase retrieval. We will now present results regarding frames with these such properties. These results stand alone but they also help in analyzing further phase retrieval and norm retrieval results.

Next we will compare phase retrieval for Parseval frames
and their Naimark complements.   
For this we need a result from \cite{BCPS}:

\begin{theorem}\cite{BCPS}\label{thmbcps}
Let $P$ be a projection on $\mathcal{H}_M$ with orthonormal basis
 $\{e_i\}_{i=1}^M$ and let $\mathcal{I}\subset \{1,2,\ldots,M\}$.
The following are equivalent:
 
 (1)  $\{Pe_i\}_{i\in \mathcal{I}}$ is linearly independent.
 
 (2)  $\{(I-P)e_i\}_{i\in \mathcal{I}^c}$ spans $(I-P)\mathcal{H}_M$.
 \end{theorem}

 First, we need to see that the full spark property passes
 from a frame to its Naimark complement.

\begin{proposition}\label{cor1314}
 A Parseval frame is full spark if and only if its Naimark
 complement is full spark.
 \end{proposition}

\begin{proof}
By Theorem \ref{naimark} any Parseval frame can be written as $\{Pe_i\}_{i=1}^M$ where $\{e_i\}_{i=1}^M$ is an orthonormal basis for $\ell_2(M)\simeq\mathcal{H}_M$ and $P$ is an orthogonal projection.  Furthermore, the Naimark complement of this Parseval frame is $\{(I-P)e_i\}_{i=1}^M$.  We also have that $\{Pe_i\}_{i=1}^M$ is full spark if and only if for any subset $\mathcal{I}\subseteq\{1,...,M\}$ such that $|\mathcal{I}|=N$ we have that $\{Pe_i\}_{i\in\mathcal{I}}$ is linearly independent and spanning (in the image of $P$).  So under this assumption Theorem \ref{thmbcps} implies that $\{(I-P)e_i\}_{i\in\mathcal{I}}$ is also linearly independent and spanning (in the image of $I-P$), so $\{(I-P)e_i\}_{i=1}^M$ is also full spark.  The other direction follows from the same argument by reversing the roles of $P$ and $I-P$.
\end{proof}

In general, if a Parseval frame yields phase retrieval, its
Naimark complement may not yield phase retrieval.  This follows
from the fact that there may not be enough vectors in the
Naimark complement to satisfy the complement property
as we see in the next proposition.

\begin{proposition}
Assume a Parseval frame $\{\varphi_i\}_{i=1}^M$ yields 
phase retrieval for $\mathbb{R}^N$ and its Naimark
complement $\{\psi_i\}_{i=1}^M$
yields phase retrieval for $\mathbb{R}^{M-N}$.
Then $2N-1 \le M \le 2N+1$.
\end{proposition}

\begin{proof}
Since $\{\varphi_i\}_{i=1}^M$ yields phase retrieval in
$\mathbb{R}^N$, we have by Theorem \ref{thmbce} that
$2N-1 \le M$.  If the Naimark complement
$\{\psi_i\}_{i=1}^M$ yields phase retrieval in 
$\mathbb{R}^{M-N}$ then again by Theorem \ref{thmbce}
we have $2(M-N)-1 \le M$. That is, $M \le 2N+1$.
\end{proof}

Unfortunately, even if we restrict the number of vectors in a Parseval frame to $2N-1 \le M \le 2N+1$, its Naimark
complement still might not yield phase retrieval.

\begin{example}
Let $\{\varphi_i\}_{i=2}^{2N}$ be a set of full spark vectors in $\mathbb{R}^N$, with $N \geq 3$. Define $\varphi_1=\varphi_2$ and let $S$ be the frame operator for $\{\varphi_i\}_{i=1}^{2N}$. Note that $\{S^{-\frac{1}{2}}\varphi_i\}_{i=2}^{2N}$ is still a full spark set of vectors. Therefore $\{S^{-\frac{1}{2}}\varphi_i\}_{i=1}^{2N}$ yields phase retrieval.  That is, for any partition $\mathcal{I}, \mathcal{I}^c \subset \{1,\dots, 2N\}$, either $\mathcal{I}$ or $\mathcal{I}^c$ has at least $N$ elements from the full spark family
$\{S^{-\frac{1}{2}}\varphi_i\}_{i=2}^{2N}$ and hence spans $\mathbb{R}^N$.

Now we will see that the Naimark complement of  $\{S^{-\frac{1}{2}}\varphi_i\}_{i=1}^{2N}$ fails phase retrieval. Partition  $\{S^{-\frac{1}{2}}\varphi_i\}_{i=1}^{2N}$ into $\{S^{-\frac{1}{2}}\varphi_i\}_{i=1}^{2}$ and $\{S^{-\frac{1}{2}}\varphi_i\}_{i=3}^{2N}$. Observe that neither set is linearly independent since $\varphi_1=\varphi_2$ and $N \geq 3$. By Theorem \ref{thmbcps}, the Naimark complement of each set does not span $\mathbb{R}^{2N-N}=\mathbb{R}^N$. Hence this is a partition of the Naimark complement of  $\{S^{-\frac{1}{2}}\varphi_i\}_{i=1}^{2N}$ which fails complement property and therefore fails phase retrieval. 
\end{example}

With the aid of full spark, we are able to pass
phase retrieval to Naimark complements as long as we
satisfy the restriction on the number of vectors.

\begin{proposition}
If $\Phi=\{\varphi_i\}_{i=1}^M$ is a full spark Parseval frame in $\mathbb{R}^N$ and $2N-1 \leq M \leq 2N+1$ then $\Phi$ yields phase retrieval in $\mathbb{R}^N$ and the Naimark complement of $\Phi$ yields phase retrieval in $\mathbb{R}^{M-N}$. 
\end{proposition}

\begin{proof}
By Proposition \ref{cor1314}, the Naimark complement of $\Phi$ is full spark in $\mathbb{R}^{M-N}$. Since $2N-1 \le M$
and $2(M-N)-1 \le M$, by Theorem  \ref{thmbce} both $\Phi$ and its Naimark complement have the complement property in their respective spaces.  \end{proof}


\section{Phase retrieval by projections and norm retrieval}

Phase retrieval for the higher dimensional case is similar to the one dimensional case. 

\begin{definition}
Let $\{W_{i}\}_{i=1}^{M}$ be a collection of subspaces of $\mathcal{H}_{N}$ and let $\{P_{i}\}_{i=1}^{M}$ be the
orthogonal  projections onto these subspaces. We say that $\{W_{i}\}_{i=1}^{M}$ (or $\{P_{i}\}_{i=1}^{M}$) yields \textbf{phase retrieval} if for all $x,y\in\mathcal{H}_{N}$ satisfying $\norm{P_{i}x}=\norm{P_{i}y}$ for all $i=1,\ldots,M$, then $x=cy$ for some scalar $c$ with $\abs{c}=1$.
\end{definition}

Recently, a detailed study of phase retrieval by projections appeared
in \cite{CCPW}. Originally, it was believed that higher ranked projections gave
much less information than vectors and hence phase retrieval
by projections would require many more projections than
phase retrieval by vectors requires.  In \cite{CCPW}, the surprising result appears that we do not need more projections to yield phase retrieval.

\begin{theorem}\cite{CCPW}
Phase retrieval can be done in $\mathbb{R}^N$ with $2N-1$ arbitrary rank projections $(0 <$ rank $P_i<N$).
\end{theorem}

This result yields the equally surprising problem,

\begin{problem}
Can phase retrieval be done in $\mathbb{R}^N$ with fewer
than $2N-1$ projections?
\end{problem}

Another question which naturally arises in this context is given subspaces $\{W_i\}_{i=1}^M$ of $\mathcal{H}_N$
which yield phase retrieval, do $\{W_i^{\perp}
\}_{i=1}^M$ yield phase retrieval?  It is shown in \cite{CCPW}
that this is not true in general.  We introduce a
new fundamental property, norm retrieval, which is
precisely what is needed to pass phase retrieval to
orthogonal complements.  Next we make precise the definition of norm retrieval and then show its importance to phase retrieval. After this we will develop the basic properties of norm retrieval.

\begin{definition}
Let $\{W_i\}_{i=1}^M$ be a collection of subspaces in $\mathcal{H}_N$ and define $\{P_i\}_{i=1}^M$ to be the orthogonal projections onto each of these subspaces. We say that $\{W_{i}\}_{i=1}^{M}$ (or $\{P_i\}_{i=1}^M$) yields \textbf{norm retrieval} if for all $x,y \in \mathcal{H}_N$ satisfying $\norm{P_{i}x}=\norm{P_{i}y}$ for all $i=1,\ldots,M$, then $\norm{x}=\norm{y}$.
\end{definition}

\begin{remark}
Although trivial, it is important to point out that any collection of subspaces which yields phase retrieval necessarily yields norm retrieval. However, the converse need not hold since any orthonormal set of vectors does norm
retrieval but has too few vectors to do phase retrieval.
\end{remark}

We will start by showing that norm retrieval is precisely
the condition needed to pass phase retrieval to orthogonal
complements.

\begin{proposition}\label{oprenr}
Let $\{W_i\}_{i=1}^M$ be a collection of subspaces in $\mathcal{H}_N$ yielding phase retrieval and let $\{P_i\}_{i=1}^M$ be the projections onto these subspaces. The following are equivalent:
\begin{enumerate}
\item $\{I-P_i\}_{i=1}^M$ yields phase retrieval.
\item $\{I-P_i\}_{i=1}^M$ yields norm retrieval.
\end{enumerate}
\end{proposition}

\begin{proof}
 $(1)$ $\Rightarrow$ $(2)$ Since phase retrieval always implies norm retrieval then this is clear.\\
 $(2)$ $\Rightarrow$ $(1)$ Let $x,y \in \mathcal{H}_N$ be such that $\|(I-P_i)x\|^2=\|(I-P_i)y\|^2$ for all $i=1,\dots,M$. Since $\{I-P_i\}_{i=1}^M$ yields norm retrieval then this implies that $\|x\|^2=\|y\|^2$. Since $P$ and $(I-P)$ correspond to orthogonal subspaces then $\|x\|^2=\|P_ix\|^2+\|(I-P_i)x\|^2$ for all $i=1,\dots,M$. Thus, for all $i=1,\dots, M$ we have
\[\|P_ix\|^2=\|x\|^2 - \|(I-P_i)x\|^2=\|y\|^2 -\|(I-P_i)y\|^2=\|P_iy\|^2.\]
Since $\{W_i\}_{i=1}^M$ yields phase retrieval then this implies that $x=cy$ for some scalar $|c|=1$. Therefore $\{I-P_i\}_{i=1}^M$ yields phase retrieval.\end{proof}

We can think of norm retrieval as giving us one free
measurement when trying to do phase retrieval.  For
example, let $\{e_i\}_{i=1}^3$ be an orthonormal basis for
$\mathbb{R}^3$ and choose $\{\varphi_1,\varphi_2\}$ so
that these 5 vectors are full spark.  Hence, these vectors
yield phase retrieval in $\mathbb{R}^3$.  Now consider the
family of vectors $\Phi:=\{e_1,e_2,\varphi_1,\varphi_2\}$ in
$\mathbb{R}^3$. $\Phi$ cannot do phase retrieval in general
since it takes at least 5 vectors to
do phase retrieval in $\mathbb{R}^3$.  However, given
unit norm $x,y\in \mathbb{R}^3$, $\Phi$ will do
phase retrieval in this scenario since we have knowledge of the norms of the signals.  Assume
\[ |\langle x,e_i\rangle|=|\langle y,e_i\rangle|
\mbox{ and } |\langle x,\varphi_i\rangle| = |\langle y,
\varphi_i\rangle|\mbox{ for } i=1,2.\]
Then 
\[ 1=\|x\|^2 = \sum_{i=1}^3|\langle x,e_i\rangle|^2,\]
and similarly for $y$. This implies
\[ |\langle x,e_3\rangle|^2=
1-\sum_{i=1}^2|\langle x,e_i\rangle|^2 = 1-\sum_{i=1}^2
|\langle y,e_i\rangle|^2 = |\langle y,e_3\rangle|^2.\]
That is, $x,y$ have the same modulus of inner products with
all 5 vectors which yield phase retrieval and so
$x = \pm y$.

A fundamental idea is to apply operators to vectors and
subspaces which yield phase retrieval or norm
retrieval.  We now consider 
when operators preserve these concepts. 

\begin{proposition}\label{prop5.4}
If $\{\varphi_i\}_{i=1}^M$ is a frame in $\mathcal{H}_N$ which yields phase retrieval (respectively norm retrieval) then $\{P\varphi_i\}_{i=1}^M$ yields phase retrieval (respectively norm retrieval) for all orthogonal projections $P$ on $\mathcal{H}_N$.
\end{proposition}

\begin{proof}
Let $x,y \in P(\mathcal{H}_N)$ such that $|\langle x, P\varphi_i\rangle|^2 = |\langle y, P\varphi_i\rangle|^2$ for all $i=\{1,\dots, M\}$. For all $i \in \{1,\dots,M\}$, we have 
\[|\langle x, \varphi_i\rangle |^2 = |\langle Px, \varphi_i\rangle|^2=|\langle x, P\varphi_i\rangle|^2=|\langle y, P\varphi_i\rangle|^2=|\langle Py, \varphi_i\rangle|^2 =|\langle y, \varphi_i\rangle|^2.\]
Since $\{\varphi_i\}_{i=1}^M$ gives phase retrieval (respectively norm retrieval) then this implies $x=cy$ for some scalar $|c|=1$ (respectively $\|x\|=\|y\|$). Therefore, $\{P\varphi_i\}_{i=1}^M$ yields phase retrieval (respectively norm retrieval).
\end{proof}

Although norm retrieval is preserved when applying any projection to the vectors, this does not hold when we apply an invertible operator to the vectors. The next theorem classifies when invertible operators
maintain norm retrieval.

\begin{theorem} \label{abc}
Let $\{\varphi_i\}_{i=1}^M$ be vectors in $\mathcal{H}_N$. The following are equivalent:
\begin{enumerate}
\item $\{\varphi_i\}_{i=1}^M$ yields phase retrieval.
\item $\{T\varphi_i\}_{i=1}^M$ yields phase retrieval for all invertible operators $T$ on $\mathcal{H}_N$.
\item $\{T\varphi_i\}_{i=1}^M$ yields norm retrieval for all invertible operators $T$ on $\mathcal{H}_N$.
\end{enumerate}
\end{theorem}

\begin{proof}
$(1) \Rightarrow (2)$ Let $T$ be any invertible operator on $\mathbb{R}^N$ and let $x,y \in \mathcal{H}_N$ be such that $|\langle x, T\varphi_i \rangle|= |\langle y, T\varphi_i \rangle|$ for all $i \in \{1,\dots,M\}$. Then $|\langle T^*x, \varphi_i\rangle|=|\langle T^*y, \varphi_i \rangle|$ for all  $i \in \{1,\dots,M\}$. Since $\{\varphi_i\}_{i=1}^M$ yields phase retrieval then this implies $T^*x=cT^*y$ for some scalar $|c|=1$. Since $T$ is invertible and linear then $T^{-*}T^*x=T^{-*}cT^*y$ implies $x = cy$ and $|c|=1$. Therefore, $\{T\varphi_i\}_{i=1}^M$ does phase retrieval.

\vspace{.1in}

$(2)\Rightarrow (3)$ Since phase retrieval implies norm retrieval then this is clear.
\vspace{.1in}

$(3) \Rightarrow (1)$ Choose nonzero $x,y \in \mathcal{H}_N$ such that $|\langle x, \varphi_i \rangle|=|\langle y, \varphi_i \rangle| $ for all $i \in \{1,\dots,M\}$. By assumption, $\{T\varphi_i\}_{i=1}^M$ yields norm retrieval for all invertible operators $T$ on $\mathcal{H}_N$. 

Let $T$ be any invertible operator on $\mathcal{H}_N$, then $T^{-*}$ is an invertible operator and hence $\{T^{-*}\varphi_i\}_{i=1}^M$ yields norm retrieval. For $Tx, Ty \in \mathcal{H}_N$, we have 
$$
\begin{array}{rcl}
|\langle Tx,T^{-*}\varphi_i\rangle| &=& |\langle T^{-1}Tx, \varphi_i\rangle|\\
&=& |\langle x, \varphi_i \rangle |\\
&=& |\langle y, \varphi_i \rangle |\\
&=&|\langle T^{-1}Ty, \varphi_i \rangle|\\
&=&|\langle Ty, T^{-*}\varphi_i\rangle |,
\end{array}
$$

for every $i \in \{1,\dots, M\}$. Hence, $||Tx||=||Ty||$ for any invertible operator $T$ on $\mathcal{H}_N$.

Now we will be done if we can show that $\|Tx\|=\|Ty\|$ for any invertible $T$ implies $y=cx$ for some scalar $c$ with $|c|=1$.  First note that since the identity operator is invertible we have that $\|x\|=\|y\|$, so if $y=cx$ then it follows that $|c|=1$.  Now choose an orthonormal basis $\{e_j\}_{j=1}^N$ for $\mathcal{H}_N$ with $e_1=\frac{x}{\|x\|}$ and suppose $y=\sum_{j=1}^N\alpha_je_j$, so that $\|y\|^2=\sum_{j=1}^N\alpha_j^2$.  Define the operator $T$ by $Te_1=e_1$ and $Te_j=\frac{1}{2}e_j$ for $j=2,...,N$.  Now we have that $\|x\|^2=\|Tx\|^2=\|Ty\|^2=\alpha_1^2+\sum_{j=2}^N\frac{1}{4}\alpha_j^2$, which implies that $\sum_{j=2}^N\alpha_j^2=\frac{1}{4}\sum_{j=2}^N\alpha_j^2$, and so $\alpha_j=0$ for $j=2,...,N$.  Therefore, $y=\alpha_1e_1=\frac{\alpha_1}{\|x\|}x$ which competes the proof.
\end{proof}

Note that the equivalence of (1) and (2) in the above Theorem was shown in \cite{BCE}, so the new part is that (3) is equivalent to both of these.

At first glance one would think that retrieving the norm of a signal would be much easier than recovering the actual signal. However, Theorem \ref{abc} gives a new classification of phase retrieval in terms of norm retrieval and states that if every invertible operator applied to a frame yields norm retrieval then our original frame yields phase retrieval. This illustrates that recovering the norm of a signal may be more similar to recovering the actual signal than originally thought and hence may not be as easily achievable as anticipated.

\begin{problem} Does Theorem \ref{abc} generalize to subspaces?
\end{problem}

Theorem \ref{abc} shows that if a frame does not yield phase retrieval, then we cannot apply an invertible operator to it in order to get a frame that does yield phase retrieval. In contrast, it is true that there exists at least one invertible operator which when applied to a frame yields norm retrieval.

\begin{proposition}
Given $\{\varphi_i\}_{i=1}^M$ spanning $\mathcal{H}_N$, there exists an invertible operator $T$ on $\mathcal{H}_N$ so that $\{T\varphi_i\}_{i=1}^M$ yields norm retrieval.
\end{proposition}

\begin{proof}
Without loss of generality, assume $\{\varphi_i\}_{i=1}^N$ are linearly independent. Choose an invertible operator $T$ so that $T\varphi_i = e_i$ for all $i=1,\dots, N$, where $\{e_i\}_{i=1}^N$ is an orthonormal basis for $\mathcal{H}_N$. Thus, for any $x \in \mathcal{H}_N$, $||x||^2=\sum_{i=1}^N | \langle x, e_i \rangle |^2 = \sum_{i=1}^N | \langle x, T\varphi_i \rangle |^2$. Hence, $\{T\varphi_i\}_{i=1}^N$ yields norm retrieval. In particular, $\{T\varphi_i\}_{i=1}^M$ yields norm retrieval. 
\end{proof}
\section{The identity operator, norm retrieval and phase retrieval}

Instead of applying operators to a frame and observing properties regarding the image; we now look at what operators lie in the span of the projections onto a collection of subspaces. In numerous results and examples regarding phase retrieval, we have found that often times when a collection of vectors or projections yields phase retrieval or norm retrieval then the identity is in their span. In this section we classify when this occurs and when it fails.  In 
\cite{CCPW}, they show that having the identity in the
span of a family of projections doing phase retrieval,
will generally yield that the orthogonal complements do
phase retrieval.

\begin{theorem}\cite{CCPW}\label{aaa}
Assume $\{W_i\}_{i=1}^M$ are subspaces of $\mathbb{R}^N$
yielding phase retrieval with corresponding 
orthogonal projections $\{P_i\}_{i=1}^M$.  If $I=\sum_{i=1}^M
a_iP_i$ and $\sum_{i=1}^M a_i \not= 1$, then $\{W_i^{\perp}
\}_{i=1}^M$ yields phase retrieval.
\end{theorem}
 
We now state one consequence of Theorem \ref{aaa} which did not appear in \cite{CCPW}.

\begin{theorem}\label{cor2}
Let $\{W_i\}_{i=1}^M$ be a collection subspaces of $\mathcal{H}_N$ that yields phase retrieval, and suppose further that $\mathrm{dim}(W_i)=K$ for every $i=1,2,...,M$.  Let $P_i$ be the orthogonal projection onto $W_i$ and suppose that $I\in\mathrm{span}\{P_i\}_{i=1}^M$, then $\{W_i^{\perp}\}$ yields phase retrieval.
\end{theorem}

\begin{proof}
Let $I=\sum_{i=1}^Ma_iP_i$.  Then
\[N=\mathrm{Tr}(I) 
=\mathrm{Tr}(\sum_{i=1}^Ma_iP_i) 
=\sum_{i=1}^Ma_i\mathrm{Tr}(P_i) 
=K\sum_{i=1}^Ma_i\]
since $\mathrm{Tr}(P_i)=K$ for every $i$.  Therefore, $\sum_{i=1}^Ma_i=\frac{N}{K}>1$ since $K<N$, so the result follows from Theorem \ref{aaa}.
\end{proof}

Given a collection of orthogonal projections $\{P_i\}_{i=1}^M$, we cannot conclude that they yield phase retrieval just because $I\in\mathrm{span}\{P_i\}_{i=1}^M$. For example, given any projection $P$, then $I=P+(I-P)$ but certainly $\{P,I-P\}$ will not yield phase retrieval.  However, we now show that this is enough to conclude that $\{P_i\}_{i=1}^M$ yields norm retrieval.

\begin{proposition}\label{nrtight} Let $\{W_{i}\}_{i=1}^{M}$ be subspaces of $\mathcal{H}_{N}$, and let $\{P_{i}\}_{i=1}^{M}$ be the associated projections. If $I\in\lspan\{P_{i}\}_{i=1}^{M}$, then $\{W_{i}\}_{i=1}^{M}$ gives norm retrieval.
\end{proposition}

\begin{proof} 
Suppose $I=\sum_{i=1}^{M}a_iP_{i} = I$. Notice for $x\in\mathcal{H}_{N}$ we have
\[\sum_{i=1}^{M}a_{i}\norm{P_{i}x}^{2} = \sum_{i=1}^{M}\langle a_{i}P_{i}x,x\rangle = \left\langle \sum_{i=1}^{M}a_{i}P_{i}x,x\right\rangle = \langle Ix,x\rangle = \norm{x}^{2}\]

Let $x,y\in\mathcal{H}_{N}$ such that $\norm{P_{i}x}=\norm{P_{i}y}$ for all $i=1,\ldots,M$. We have
\[\norm{x}^{2} = \sum_{i=1}^{M}a_{i}\norm{P_{i}x}^{2} = \sum_{i=1}^{M}a_{i}\norm{P_{i}y}^{2} = \norm{y}^{2}.\]
Hence $\{W_{i}\}_{i=1}^{M}$ yields norm retrieval.
\end{proof}

The converse of Proposition \ref{nrtight} is far from true.  One way to see this (at least for the real case) is as follows:  For $2N\leq M\leq N(N+1)/2$ choose any full spark frame $\Phi=\{\varphi_i\}_{i=1}^M$ for $\mathbb{R}^N$ such that $\{\varphi_i\varphi_i^*\}_{i=1}^M$ is linearly independent (which happens generically, see e.g., \cite{CC}).  Let $S$ be the frame operator for $\Phi$ and define
$$
\psi_i=\frac{S^{-1/2}\varphi_i}{\|S^{-1/2}\varphi_i\|}
$$
with $P_i=\psi_i\psi_i^*$ (note that $P_i$ is a rank one orthogonal projection).  Since $\{S^{-1/2}\varphi_i\}_{i=1}^M$ is a Parseval frame it follows that
\begin{equation}\label{pop}
I=\sum_{i=1}^M\|S^{-1/2}\varphi_i\|^2P_i.
\end{equation}
Also, since $\{\varphi_i\}_{i=1}^M$ is linearly independent it follows that $\{P_i\}_{i=1}^M$ is linearly independent (and so \eqref{pop} is the only way to write $I$ as a linear combination of the $P_i$'s). Also since $\{\varphi_i\}_{i=1}^M$ is full spark we know that $\|S^{-1/2}\varphi_i\|\neq 0$ for every $i=1,2,...,M$.  Therefore it follows that if $\mathcal{I}\subseteq\{1,2,...,M\}$ then $I\not\in\mathrm{span}\{P_i\}_{i\in\mathcal{I}}$.  Furthermore, since $M\geq 2N$ and $\{\varphi_i\}_{i=1}^M$ is full spark it follows that $\{P_i\}_{i\in\mathcal{I}}$ yields phase retrieval (and hence norm retrieval) whenever $|\mathcal{I}|\geq 2N-1$.

Although the above example
 proves that the converse of Proposition \ref{nrtight} is false, in the special case where $\sum_{i=1}^M\mathrm{dim}(W_i)=N$ it turns out to be true.

\begin{proposition}\label{orthogonal}
A collection of unit norm vectors $\{\varphi_i\}_{i=1}^N$ in $\mathcal{H}_{N}$ yield norm retrieval if and only if $\{\varphi_i\}_{i=1}^N$ are orthogonal.
\end{proposition}

Proposition \ref{orthogonal} is a consequence of the following more general theorem about subspaces.

\begin{theorem}\label{nrbasis}
Let $\{W_i\}_{i=1}^M$ be a collection of subspaces of $\mathcal{H}_N$ with the property that $\sum_{i=1}^M \dim(W_i)=N$ and let $P_i$ be the orthogonal projection onto subspace $W_i$ for each $i=1,\dots, M$. The following are equivalent:
\begin{enumerate}
\item $\{W_i\}_{i=1}^M$ yields norm retrieval 
\item $\sum_{i=1}^M P_i=I$.
\end{enumerate}
\end{theorem}

\begin{proof} 
$(2) \Rightarrow (1)$ Follows from Proposition \ref{nrtight}.

$(1) \Rightarrow (2)$ Pick some $W_j$ and define $V_j$ to be the span of the $M-1$ subspaces $\{W_i\}_{i\neq j}$, and let $Q$ be the orthogonal projection onto $V_j$. Without loss of generality we may assume that $W_j$ is not the zero subspace. Note that 
\[\dim V_j\leq \sum_{i\neq j}\dim(W_{i})=N-\dim(W_j).\]

\emph{Claim 1:} $W_j \cap V_j= \{0\}$.

\emph{Proof of Claim:} 
Assume to the contrary that $W_{j}\cap V_j$ is nontrivial. Then
\[\dim\lspan\{W_{i}\}_{i=1}^{M} < \dim V_j+\dim W_j\leq N.\]
This implies that there exists a nonzero $x_0\in(\lspan\{W_{i}\}_{i=1}^{M})^{\perp}$ and hence $P_{i}x_0=0$ for all $i=1,\ldots,M$. However, since $\{W_{i}\}_{i=1}^{M}$ gives norm retrieval, we conclude that $x_0=0$, a contradiction. Thus $W_j\cap V=\{0\}$.\\

\emph{Claim 2: $P_jQ=QP_j=0$.}

\emph{Proof of Claim 2:}
Assume toward a contradiction that $P_j Q \neq 0$, and thus $QP_j\neq 0$. Set $Y=\{x\in W_j\colon Qx=0\}$. Since $QP_j\neq 0$ we see that $Y\neq W_j$. 

Let $Z$ be the orthogonal complement of $Y$ in $W_j$. Since $V_j\cap W_j=\{0\}$ and $Z\subset W_j$, we conclude that $V_j\cap Z=\{0\}$.

Let $z\in Z\setminus \{0\}$. Since $z\notin V_j$ we have $Qz\neq z$, and since $z\notin Y$ we have $Qz\neq 0$. Set $x:=Qz\neq0$ (note $x \neq z$) and $y:=(I-Q)z\neq0$.

Note that
\begin{align*}
    \langle P_j x,y\rangle &=\langle P_j Qz,(I-Q)z\rangle = \langle P_j Qz,P_j(I-Q)z\rangle \\
    &= \langle P_j Qz,P_jz-P_jQz\rangle = \langle P_j Qz,z-P_jQz\rangle\\
    &=\langle P_j Qz,z\rangle-\|P_jQz\|^2=\langle Qz,P_jz\rangle-\|P_jQz\|^2\\
		&=\langle Qz,z\rangle-\|P_jQz\|^2 =\langle Qz,Qz\rangle-\|P_jQz\|^2\\
		&= \|Qz\|^2-\|P_jQz\|^2
    \end{align*}

		\emph{Subclaim:} $\langle P_jx,y\rangle$ is nonzero and positive.
		
		\emph{Proof of Subclaim:} If $\langle P_jx,y\rangle=0$, then $\|Qz\|^2=\|P_jQz\|^2$. Note that $\|Qz\|^2=\|P_jQz\|^2+\|(I-P_j)Qz\|^2$ and hence $\|P_jz\|^2 \leq \|Qz\|^2$. This forces $\langle P_j x,y\rangle = \|Qz\|^2-\|P_jQz\|^2 \geq 0$.  These facts imply that $P_jQz=Qz$ and hence $Qz=x \in W_j$, which contradicts the fact that $W_j\cap V_j=\{0\}$. Thus, $\langle P_jx,y\rangle$ is nonzero and positive.

\vspace{.2in}

Set $v_{1}=x$ and $v_{2}=x + \alpha y$
for some $\alpha\in\mathcal{H}$ which will be specified later. Since $P_{i}(I-Q)=0$ for $i\neq j$, we have 
\begin{align*}
\norm{P_{i}v_{2}} &=\|P_i(x + \alpha y)\|=\|P_i(Qz+ \alpha (I-Q)z)\|\\
&=\|P_i Qz + \alpha P_i(I-Q)z\|= \|P_iQz\|=\|P_ix\|\\
&=\norm{P_{i}v_{1}}
\end{align*}
 for all $i\neq j$.

If $P_jy=0$ then we take $\alpha$ to be any nonzero scalar in $\mathcal{H}$, and we have \[\norm{P_jv_{1}}=\norm{P_jx}= \|P_jx+\alpha P_jy\| = \norm{P_jv_{2}}.\]  

If $P_jy\neq 0$ then we set $\alpha=-\frac{2\langle P_jx,y\rangle}{\norm{P_jy}^{2}}$ and we have
\begin{align*}
\norm{P_jv_{2}}^{2} & = (P_jx + \alpha P_jy)^2\\
&=
\norm{P_jx}^{2} + \alpha \bar{\alpha}\|P_jy\|^2 + \bar{\alpha}\langle P_jx, P_jy \rangle + \alpha \langle P_j y, P_j x\rangle\\
&= 
\norm{P_jx}^{2} + \frac{4\langle P_jx,y\rangle \langle y,P_jx\rangle}{\|P_jy\|^4} \|P_jy\|^2 - \frac{2\langle y,P_jx \rangle \langle P_jx, y\rangle}{\|P_jy\|^2} - \frac{2\langle P_jx, y \rangle \langle y, P_jx\rangle}{\{P_jy\|^2}\\
& = 
\norm{P_jx}^{2}+\frac{4|\langle P_jx,y\rangle|^{2}}{\norm{P_jy}^{2}}-\frac{2|\langle P_jx,y\rangle|^2}{\norm{P_jy}^{2}}-\frac{2|\langle P_jx,y\rangle|^2}{\norm{P_jy}^{2}} \\
&= \norm{P_jx}^{2} = \norm{P_jv_{1}}^{2}.
\end{align*}

Thus $\|P_iv_1\|=\|P_iv_2\|$ for all $i = 1,\dots, M$. However, for any $\alpha\in\mathcal{H}\setminus\{0\}$ we have
\[\norm{v_{2}}^{2} = \norm{x}^{2} + |\alpha|^2\norm{y}^{2}>\norm{x}^{2} = \norm{v_{1}}^{2}.\]
Hence $\{W_{i}\}$ does not yield norm retrieval, a contradiction to our assumption, and so $QP_j=P_jQ=0$, which finishes the proof of Claim 2. 

Therefore, we have that
$$
W_j=\mathrm{im}(P_j)=\mathrm{ker}(Q)=V_j^{\perp}.
$$
But if $i\neq j$ then $W_i\subseteq V_j$ and so $W_i\perp W_j$, from which (2) easily follows.
\end{proof}



\bibliographystyle{amsalpha}

\end{document}